\documentclass[12pt]{article}
\usepackage{GenericStyle}

\title{Convergence to the uniform distribution of moderately self-interacting diffusions on compact Riemannian manifolds\footnote{This version differs from the accepted AIHP version only by minor notational adjustments. A last-minute oversight had led to a confusing presentation of equation \eqref{eq:limitODEnewreal}, conflating it with equation \eqref{eq:limitODEnew}, which is newly introduced here for clarity.}}
\author{Simon Holbach\footnote{s.holbach@uni-mainz.de}\qquad Olivier Raimond\footnote{oraimond@parisnanterre.fr}}

\begin{document}

\maketitle

\begin{abstract}
We consider a self-interacting diffusion $X$ on a smooth compact Riemannian manifold $\M$, described by the stochastic differential equation
\[
dX_t = \sqrt{2} dW_t(X_t)- \beta(t) \nabla V_t(X_t)dt,
\]
where $\beta$ is suitably lower-bounded and grows at most logarithmically, and $V_t(x)=\frac{1}{t}\int_0^t V(x,X_s)ds$ for a suitable smooth function $V\colon \M^2\to\R$ that makes the term $-\nabla V_t(X_t)$ self-repelling. We prove that almost surely the normalized occupation measure $\mu_t$ of $X$ converges weakly to the uniform distribution $\cU$, and we provide a polynomial rate of convergence for smooth test functions. The key to this result is showing that if $f\colon\M\to\R$ is smooth, then $\mu_{e^t}(f)$ shadows the flow generated by the ordinary differential equation
\[
    \dot x_t=-x_t+\cU(f).
\]
\end{abstract}

\section{Introduction and main result}

Let $\M$ a smooth compact Riemannian manifold, and let $\cM(\M)$ denote the space of finite signed Borel measures on $\M$. For any smooth function $V\colon \M^2\to\R$ and any $\mu\in \cM(\M)$ we write
\begin{equation}\label{eq:V_mu}
	V_\mu(x):=\mu(V(x,\cdot)):=\int_\M V(x,y)\mu(dy) \quad \text{for all $x\in\M$.}
\end{equation}
A stochastic process $X$ is called a self-interacting diffusion on $\M$, if it satisfies an equation of the type
\begin{equation}\label{eq:selfinteractingM}
	dX_t = \sqrt{2} dW_t(X_t)- \beta(t) \nabla V_{\mu_t}(X_t)dt,
\end{equation}
where $V\colon \M^2\to\R$ is smooth, $\nabla$ denotes the surface gradient on $\M$, $\beta\colon[0,\infty)\to\R$ is a continuous function, 
\begin{equation*}
    \mu_t=\frac 1 t \int_0^t\delta_{X_s}ds
\end{equation*}
is the normalized occupation measure of the process $X$ up until time $t$, and $W$ is a standard Brownian vector field on $\M$.

Here, $V$ describes the type of self-interaction and $\beta$ is the temporal weight that is given to the self-interaction mechanism. The asymptotic behavior of $X$ has been studied for different choices of $V$ and $\beta$, and we provide a summary of corresponding results and methods in Section \ref{sec:context} below.

Let us now describe the assumptions under which we work in the present article. We always write $|x|$ for the euclidean norm of $x\in\R^d$, by $x \cdot y$ we denote the euclidean inner product of $x,y\in\R^d$, and we set
\[
    \norm{f}_\infty=\sup_{x\in\M}\abs{f(x)} 
\]
for any function $f\colon \M\to\R^d$.

\begin{ass}\label{ass:V}
There is a smooth function $v\colon \M\to\R^N$ with $\norm{v}_\infty=1$ and
\begin{equation}\label{eq:assvcentered}
    \int_\M v(x) dx=0
\end{equation}
such that
\[
    V(x,y)=v(x)\cdot v(y) 
\]
for all $x,y\in\M$.
\end{ass}

With this choice of $V$, \eqref{eq:V_mu} yields
\begin{equation*}\label{eq:V_t}
    V_{\mu_t}(x)= \mu_t(v) \cdot v(x) = \left(\frac 1 t \int_0^t v(X_s)ds\right) \cdot v(x).
\end{equation*}
Hence, the drift term $-\nabla V_{\mu_t}(X_t)$ is self-repelling in the sense that it tends to drive $v(X_t)$ away from the temporal mean of $(v(X_s))_{s\in[0,t]}$. This interpretation is particularly intuitive in the case where
\[
    \M=\S^n=\{x\in\R^{n+1}:\abs{x}=1\}
\]
and $v(x)=x$, so that
\[
    V(x,y)= x\cdot y=\cos(d(x,y))  \quad \text{for all $x,y\in\S^n$,}
\]
where $d$ is the geodesic distance on $\S^n$.

Next, we explain our assumptions on $\beta$. Here and everywhere else in this article, $C$ and $t_0$ denote finite, positive, deterministic constants, the exact value of which is unimportant and may change from one step to the next with no indication.

\begin{ass}\label{ass:beta}
The function $\beta\colon [0,\infty)\to\R$ is differentiable and there are $a\in(0,\infty)$ and $\gamma\in(0,1]$ such that
\begin{equation*}
    \abs{\beta(t)}\le a\log t \quad \text{and} \quad 
    \abs{\beta'(t)}\le C t^{-\gamma} \quad \text{for all $t\ge t_0$.}
\end{equation*}
\end{ass}

Assumption \ref{ass:beta} allows the weight $\beta(t)$ of the self-repelling drift $-\nabla V_{\mu_t}$ to increase to infinity, but not fast enough to fully compensate the normalization $\frac 1 t$ of the occupation measure. The generic case that we will usually have in mind is that of
\[
    \beta(t)=b\log(t+1) \quad \text{with $b>0$.}
\]
Other valid choices include $\beta(t)=b\log(\log(t+e))$ or simply $\beta\equiv b$. Note that we cap $\gamma$ at 1 only for simplicity, as $\gamma>1$ provides no meaningful improvement for the estimates that are relevant to our proofs and only results in awkward case distinctions. Also note that the normalization $\norm{v}_\infty=1$ in Assumption \ref{ass:V} makes sure that the parameter $a$ is actually meaningful and cannot simply be hidden in $v$.

Assumption \ref{ass:beta} does not require $\beta$ to be non-negative, but we will assume a suitable lower bound in Assumption \ref{ass:betaLambda} below. Before we can state it precisely, we need to introduce some notation. Let $m\in\R^N$. Setting
\begin{equation*}
    Z(m)=\int_\M e^{-m\cdot v(x)}dx,
\end{equation*}
we define a probability measure $\Pi(m)$ on $\M$ via
\begin{equation}\label{eq:PIfirst}
    \Pi(m)(dx)=\frac{e^{-m\cdot v(x)}}{Z(m)}dx.
\end{equation}
We also interpret $\Pi$ as a function on $\cM(\M)$ by setting
\begin{equation}\label{eq:PIfunctionalmeasure}
    \Pi(\mu):=\Pi(\mu(v)) \quad \text{for all $\mu\in\cM(\M)$.}
\end{equation}
Furthermore, for any probability measure $\mu$ on $\M$ and any $f\in(\LL^2(\mu))^N$ we write
\[
    \Cov_\mu(f)=\big(\mu(f_if_j)-\mu(f_i)\mu(f_j)\big)_{i,j\in\{1,\ldots,N\}}\in\R^{N\times N}.
\]

\begin{ass}\label{ass:betaLambda}
There is a $\beta_0\ge0$ such that
\begin{equation}\label{eq:betalowerbound}
    \beta(t)\ge -\beta_0 >-\frac{1}{\Lambda} \quad \text{for all $t\ge t_0$,}
\end{equation}    
where
\begin{equation*}
    \Lambda=\sup_{m\in\R^N}\left(\sup_{x\in\R^N, \, \abs{x}=1} x^T\Cov_{\Pi(m)}(v)x\right).
\end{equation*}
\end{ass}

Assumption \ref{ass:betaLambda} means that while the weight factor in front of the self-repelling drift $-\nabla V_{\mu_t}$ is in fact allowed to turn negative, we do require a suitable lower bound that depends on $v$. In other words, we allow a limited amount of self-attraction. Clearly,
\begin{equation}\label{eq:Lambdatrivialbound}
    0<\Lambda \le \norm{v}_\infty^2=1,
\end{equation}
so \eqref{eq:betalowerbound} makes sense. In particular, $\beta\equiv b$ is covered in our setting if and only if $b > - 1/\Lambda$.

Our main result is the following. 

\begin{thm}\label{thm:main}
    There is a finite positive constant $\kappa$ that depends only on the dimension of $\M$, such that the following holds: if Assumptions \ref{ass:V}, \ref{ass:beta} and \ref{ass:betaLambda} hold and the constants $a$ and $\gamma$ from Assumption \ref{ass:beta} satisfy
    \[
        \gamma>2a\kappa,
    \]
    then for all $f\in C^\infty(\M)$ we have
    \begin{equation}\label{eq:thm}
        \limsup_{t\to\infty} \frac{\log \abs{(\mu_t-\cU)(f)}}{\log t} \le -\eta \quad \text{almost surely,}
    \end{equation}
    where $\cU$ denotes the uniform distribution on $\M$ and
    \begin{equation*}
        \eta=\min\left\{\frac{\gamma}{2}-a\kappa, 1-\Lambda\beta_0\right\}>0.
    \end{equation*}
\end{thm}

\begin{remark}
By density of $C^\infty(\M)$ in $C(\M)$ with respect to $\norm{\cdot}_\infty$, \eqref{eq:thm} implies that weak convergence of $\mu_t$ to $\cU$ holds almost surely, i.e. we have the ergodicity property
\[
    \frac1t\int_0^tf(X_s)ds = \mu_t(f) \xrightarrow{t\to\infty} \cU(f) \quad \text{almost surely for all $f\in C(\M)$.}
\]
Note, however, that the speed of convergence from \eqref{eq:thm} does not carry over to all $f\in C(\M)$ via a straight forward density argument (there is no uniformity in $f\in C^\infty(\M)$ or anything similar in \eqref{eq:thm}).
\end{remark}

\begin{remark}
The constant $\kappa$ that appears in Theorem \ref{thm:main} is the same one as in \cite[Lemma 2.3]{Raimond}, where its exact value is not given. However, one can go through its proof in \cite{Raimond} and check that it is possible to take $\kappa = 2(n+3)$, with $n$ the dimension of $\mathbb{M}$ (note in particular, that the integral in the expression for $\kappa_1$ on page 254 of \cite{Raimond} can be calculated, giving $\kappa_1=n$). Since explaining the background to this constant requires the introduction of some objects that will be used in the proof of Theorem \ref{thm:main}, we defer further discussion to Lemma \ref{lem:Qt} and Remark \ref{rem:kappa} below.
\end{remark}

\begin{remark}
It could be possible to relax Assumption \ref{ass:V} to $V$ being a Mercer kernel, i.e. $V(x,y)=V(y,x)$ for all $x,y\in\M$ and
\[
    \int_\M \int_\M V(x,y) f(x) f(y) dxdy \ge 0  \quad \text{for all $f\in\LL^2(dx)$},
\]
where $dx$ denotes integration with respect to the Riemannian measure (compare \cite[p.\ 2]{BenaimRepel}). Then, by Mercer's Theorem (see \cite[Theorem 3.a.1]{Koenig}), there are an orthonormal basis $(e_n)_{n\in\N}$ of $\LL^2(dx)$ and a sequence $(\lambda_n)_{n\in\N}\subset[0,\infty)$ such that
\begin{equation}\label{eq:Mercer}
    V(x,y)=\sum_{n\in\N}\lambda_{n}e_{n}(x)e_{n}(y) \quad \text{for all $x,y\in\M$}.   
\end{equation}
Notable examples of Mercer kernels can be found in \cite[Section 2.3]{sidIII}). In particular, if $\M$ is a submanifold of $\R^d$ and $v\colon [0,\infty)\to\R$ is completely monotonic, then $V(x,y)=v(|x-y|^{2})$, $x,y\in\M$, is a Mercer kernel.

Assumption \ref{ass:V} means that we restrict ourselves to such $V$ where the expansion in \eqref{eq:Mercer} is finite (an assumption that is also used in \cite{BenaimRepel}). This makes it possible to define $\Pi$ as a function on $\R^N$ in \eqref{eq:PIfirst}, and hence to argue as we do in Lemmas \ref{lem:Pi} to \ref{lem:gradient}.
\end{remark}

\begin{example}[Weak self-interaction]\label{ex:weak}
    Let
    \[
        \beta\equiv b>-\frac{1}{\Lambda}.
    \]
    Depending on the sign of $b$, this corresponds either to a self-repelling ($b>0$) or self-attracting ($b<0$) diffusion, and since $\beta$ is constant in time, we speak of weak self-interaction.
    
    Then $\beta'\equiv 0$ and for any $a>0$ we have $|\beta(t)| \le a\log t$ for all $t\ge t_0=e^{b/a}$. Hence, Theorem \ref{thm:main} can be applied with $\gamma=1$, $\beta_0=\max\{0,-b\}$ and any $a>0$, so for all $f\in C^\infty(\M)$ we get
    \begin{equation*}
        \limsup_{t\to\infty} \frac{\log \abs{(\mu_t-\cU)(f)}}{\log t} \le - \min\left\{\frac12, 1- \Lambda |b|\right\} \quad \text{almost surely.}
    \end{equation*}
    In the particular case $\M=\S^n$ and $v(x)=x$, this strengthens part (i) of \cite[Theorem 4.5]{sidI}. In Section \ref{sec:sphere}, we provide a more detailed investigation of this connection and also a proof that we can actually weaken Assumption \ref{ass:betaLambda} in this situation (see Proposition \ref{prop:weaksphere}).
\end{example}

\begin{example}[Moderate self-repulsion]
    Let
    \[
        \beta(t)=b\log(t+1) \quad \text{with $b>0$.}
    \]
    In this situation, $\beta$ is positive and increases to infinity, but more slowly than the normalization factor $t$ in $\mu_t$, so we speak of moderate self-repulsion.
    
    Then $|\beta'(t)|\le Ct^{-1}$ and for any $a>b$ there is a $t_0>0$ such that $|\beta(t)| \le a\log t$ for all $t\ge t_0$. Hence, Theorem \ref{thm:main} can be applied whenever $b<\frac{1}{2\kappa}$, and for all $f\in C^\infty(\M)$ we get
    \begin{equation*}
        \limsup_{t\to\infty} \frac{\log \abs{(\mu_t-\cU)(f)}}{\log t} \le - \left(\frac12-b\kappa\right) \quad \text{almost surely.}
    \end{equation*}
\end{example}

\subsection{Context}\label{sec:context}

The asymptotic behavior of self-interacting diffusions has been studied with various degrees of generality concerning the state space $\M$ and the type of self-interaction governed by $V$. The weight $\beta(t)$ is usually chosen as either $\beta\equiv b$ or $\beta(t)=b t$ with $b>0$, which are sometimes referred to as weak and strong self-interaction respectively. In this sense, the prototypical case $\beta(t)=b\log(t+1)$ of Assumption \ref{ass:beta} corresponds to moderate self-interaction.

There are a number of case studies and some general results for $\M=\R^n$ (e.g.\ \cite{Cranston,RaimondOld,Herrmann,Chambeu}), but our focus lies on compact state spaces. Some of the results in the following summary are more general than others, but all of them are valid for $\M=\S^n$ and include the important case $V(x,y)=\cos(d(x,y))$ (self-repulsion) or $V(x,y)=-\cos(d(x,y))$ (self-attraction).\footnote{Of course, it is somewhat arbitrary to include the sign that distinguishes repulsion from attraction in $V$ instead of $\beta$, but this choice reinforces the interpretation of $V$ as the type of interaction and $\beta$ as a weight (even though our Assumptions \ref{ass:beta} and \ref{ass:betaLambda} do not require $\beta$ to be non-negative).} The general expectation is that a diffusion with sufficient self-attraction will asymptotically be concentrated around (or even converge to) some limit random variable $X_\infty\in\M$, while a self-repelling diffusion (on a compact state space) will quite contrarily be uniformly distributed in the limit.

We will now give a brief overview of the methods that have been used in these different cases, the corresponding results are summarized in Table \ref{table:overview}.

The case of weak self-interaction has been studied the most thoroughly, in particular in the series of papers \cite{sidI,sidII,sidIII,sidIV}. Under mild conditions on $V$, the authors link $\mu_t$ to a measure-valued ordinary differential equation and use this to precisely describe its asymptotic behavior for some specific choices of $V$. The same approach is used in \cite{Raimond} to treat the case of moderate self-interaction. This turns out to be more delicate and only the self-attracting case is solved satisfyingly. The case of moderate self-repulsion is the content of the present paper, and we use a similar strategy. A detailed explanation of the general idea of this method is presented in Section \ref{sec:outline} below, including a discussion of the differences between \cite{sidI}, \cite{Raimond}, and the present paper. The case of strong interaction has been studied with different methods. In \cite{BenaimRepel}, the authors rewrite \eqref{eq:selfinteractingM} as a time-homogeneous proper stochastic differential equation for an extended variable $(X_t,Y_t)\in\S^n\times\R^d$ and prove that in the self-repelling case it is Harris recurrent and exponentially ergodic, where the invariant distribution in restriction to $X_t$ is the uniform distribution. Almost sure convergence in the self-attracting case is proved in \cite{Gauthier}, again with arguments that involve the shadowing of an ordinary differential equation, but in a completely different way than in \cite{sidI}, \cite{Raimond}, or the present paper.

Of course, it seems plausible that increasing the strength of the repulsion or attraction by increasing the weight $\beta$ should not change the results qualitatively. If $\tilde\beta$ is asymptotically larger than $\beta$, and the self-repelling diffusion with weight $\beta$ is asymptotically uniformly distributed, then the same should be true for the self-repelling diffusion with weight $\tilde \beta$. Similarly, if the self-attracting diffusion with weight $\beta$ converges to a random variable $X_\infty$ in some sense, then the self-attracting diffusion with weight $\tilde\beta$ should converge to some $\tilde X_\infty$. However, such comparison theorems are not available, and they do not seem to be within reach. In view of these considerations, it also seems counter-intuitive that we need $a$ to be sufficiently small in Theorem \ref{thm:main} and that the rate of convergence decreases when $a$ increases. This can be thought of as a technical assumption that is an "artifact" of our method. 

\begin{table}
    \centering
    \begin{tabular}{ |c||c|c|c| } \hline
        $\M=\S^n$ &
        \makecell{strong \\[1mm] \textover[c]{$\beta(t)=bt$}{$\beta(t)=b\log(t+1)$}} &
        \makecell{moderate \\[1mm] $\beta(t)=b\log(t+1)$} &
        \makecell{weak \\[1mm] \textover[c]{$\beta(t)=b$}{$\beta(t)=b\log(t+1)$}} \\ \hline\hline
        \makecell{self-attraction \\[1mm] $V(x,y)=-\cos(d(x,y))$} &
        \makecell{$X_t\to X_\infty$ a.s. \\[1mm] (\cite{Gauthier})} &
        \makecell{$\mu_t\xrightarrow{w} \delta_{X_\infty}$ a.s. \\[1mm] (\cite{Raimond})} &
        \makecell{$\mu_t\xrightarrow{w} \mu_\infty(b,n)$ a.s. \\[1mm] (\cite{sidI,sidII,sidIII,sidIV}, this paper)} \\ \hline
        \makecell{self-repulsion \\[1mm] $V(x,y)=\cos(d(x,y))$} &
        \makecell{$\mu_t\xrightarrow{w}\cU$ a.s. \\[1mm] (\cite{BenaimRepel,BenaimRepelInfty} for $n=1$)} &
        \makecell{$\mu_t\xrightarrow{w}\cU$ a.s. \\[1mm] (this paper) } &
        \makecell{$\mu_t\xrightarrow{w}\cU$ a.s. \\[1mm] (\cite{sidI,sidII,sidIII,sidIV}, this paper)} \\ \hline
    \end{tabular}
    \caption{An (incomplete) overview of the literature on the asymptotics of self-interacting diffusions on $\S^n$, assuming $b>0$. In the case of weak self-attraction, we have $\mu_\infty(b,n)=\cU$ if $b\le n+1$, and else $\mu_\infty(b,n)$ is a Gaussian distribution centered around a random $X_\infty\in\S^n$, but with a deterministic variance that depends on $b$ and $n$. The present paper is primarily focussed on moderate and weak self-repulsion, but also contains a partial improvement of this result on weak self-attraction. In contrast to the other citations in these cases, the present paper also studies the speed of convergence.}
    \label{table:overview}
\end{table}

\begin{remark}\label{rem:sqrt2}
    The factor $\sqrt{2}$ in front of $dW_t(X_t)$ in \eqref{eq:selfinteractingM} is absent both in \cite{Raimond} (where this equation is mentioned explicitly only in the abstract) and in \cite{sidI} (where the corresponding equation is the first one of the article).
    
    In \cite{Raimond} this factor $\sqrt{2}$ is hidden in the vector fields $e_i$, so that \eqref{eq:selfinteractingM} is entirely equivalent to \cite[(1)]{Raimond} and the results from \cite{Raimond} are compatible with our setting with no adjustment of the parameters.
    
    In \cite{sidI} however, this is not the case and this factor $\sqrt{2}$ leads to a factor $2$ in \cite[(11)]{sidI} when compared to our definition of $\Pi$ in \eqref{eq:PIfunctionalmeasure}, and also to a factor $1/2$ in the definition of $A_\mu$ just below \cite[(2)]{sidI} when compared to our definition in \eqref{eq:At}. This has to be taken into account when comparing our results with those from \cite{sidI}.
\end{remark}

\subsection{Outline of proof}\label{sec:outline}

In order to get a grip of the long time behavior of
\[
    \mu_t=\frac 1 t \int_0^t\delta_{X_s}ds,
\]
we calculate its time evolution. First, we have
\[
     \partial_t\mu_t=\frac1t\left(-\mu_t+\delta_{X_t}\right),
\]
where the derivative is to be understood as the derivative of a real function pointwise in all $f\in C(\M)$. In order to eliminate the factor $\frac1t$, we look at the dynamics on an exponential time scale, i.e.\
\begin{equation*}
    \partial_t\mu_{e^t}=-\mu_{e^t}+\delta_{X_{e^t}}.
\end{equation*}
If $t>0$ is large, the distribution of $X_t$ should be close to the equilibrium of the current drift potential $\beta(t) V_{\mu_t}(x)$, i.e.\ to the probability distribution $\Pi(\beta(t)\mu_t)$ as defined in \eqref{eq:PIfunctionalmeasure}. If our intuition of the process is correct, $\Pi(\beta(t)\mu_t)$ on the other hand should be close to the uniform distribution $\cU$ on $\M$. Therefore, we set
\begin{equation}\label{eq:def-eps1}
    \eps^1_t:= \delta_{X_{e^t}}-\Pi(\beta(e^t)\mu_{e^t})
\end{equation}
and
\begin{equation}\label{eq:def-eps2}
    \eps^2_t:= \Pi(\beta(e^t)\mu_{e^t})-\cU,
\end{equation}
so that
\begin{equation}\label{eq:ODE+eps}
    \partial_t\mu_{e^t}=-\mu_{e^t}+\cU + \eps^1_t + \eps^2_t.
\end{equation}
The main idea is that if $\eps^1_t$ and $\eps^2_t$ are asymptotically negligible in a suitable sense, then the trajectory $t\mapsto\mu_{e^t}$ should in some sense shadow the flow of the formal limit equation
\begin{equation}\label{eq:limitODEnew}
    \dot\nu_t=-\nu_t+\cU.
\end{equation}
Indeed, we can show that for all $f\in C^\infty(\M)$, the trajectory $t\mapsto\mu_{e^t}(f)$ is almost surely an asymptotic pseudotrajectory of the real-valued ordinary differential equation.
\begin{equation}\label{eq:limitODEnewreal}
    \dot x_t=-x_t+\cU(f),
\end{equation}
so that in particular $\mu_{e^t}(f)$ almost surely converges to $\cU(f)$. The details of this argument are given in Section \ref{sec:mainproof} below, but first we devote Sections \ref{sec:eps1} and \ref{sec:eps2} to the required analysis of the asymptotics of $\eps^1$ and $\eps^2$.

Our proof strategy is inspired by \cite{sidI} and \cite{Raimond}. In these works, the authors make use of the relation
\begin{equation}\label{eq:ODE+eps1}
    \partial_t\mu_{e^t}=-\mu_{e^t}+\Pi(\beta(e^t)\mu_{e^t}) + \eps^1_t,
\end{equation}
which will also be very useful for us in Section \ref{sec:eps2}. The formal limit equation corresponding to \eqref{eq:ODE+eps1} is
\begin{equation}\label{eq:limitODEold}
    \dot \nu_t=-\nu_t+\Pi(\beta(e^t)\nu_t).
\end{equation}
Note that under \eqref{eq:assvcentered} we have $\Pi(\cU)=\cU$, and so \eqref{eq:limitODEold} can be viewed as a variant of the limit equation \eqref{eq:limitODEnew} that is less specific to a particular situation. For constant $\beta$ as in \cite{sidI}, the (measure-valued) equation \eqref{eq:limitODEold} is homogeneous in time and hence after establishing that $\mu_{e^t}$ shadows it, powerful results from the theory of dynamical systems can be used to study the long-time behavior of $\mu_t$ for several different choices of $V$ (cf. \cite[Sections 3 and 4]{sidI}). For non-constant $\beta$ as in \cite{Raimond}, this link can still be established under certain conditions (cf. \cite[Theorems 2.1 and 2.2]{Raimond}), but it is not as fruitful, since \eqref{eq:limitODEold} is no longer homogeneous in time. In particular, \cite[Theorems 2.1 and 2.2]{Raimond} are not used in the proof of the convergence result for moderately self-attracting diffusions on the sphere (\cite[Theorem 3.1]{Raimond}) which is instead proved "by hand".

The limit equation \eqref{eq:limitODEnew} is much simpler than \eqref{eq:limitODEold}, because it is tailor-made for cases in which we expect $\mu_t$ to converge to the uniform distribution (while in other cases we might not expect a deterministic limit at all). Introducing the second error term $\eps^2$ allows us to move the explicit dependence of the problem on $\beta$ entirely into the error terms. 

While many techniques in the present article are extensions of the works in \cite{Raimond} and \cite{sidI}, its main innovations lie in replacing \eqref{eq:limitODEold} with \eqref{eq:limitODEnew} and in the arguments in Section \ref{sec:eps2} that are necessary to deal with the extra error term that is caused by this approach. This approach allows us to cover for the first time the case of moderate self-repulsion (\cite{Raimond} only deals with one particular example of moderate self-attraction), and it also comes with the benefit of providing a speed of convergence.

\section{Proof of the main result}\label{sec:proof}

For Section \ref{sec:eps1}, we only need the Assumptions \ref{ass:V} and \ref{ass:beta}. In Sections \ref{sec:eps2} and \ref{sec:mainproof} on the other hand, we suppose that all of the Assumptions \ref{ass:V}, \ref{ass:beta}, and \ref{ass:betaLambda} hold.

\subsection{Dealing with $\eps^1$}\label{sec:eps1}

The aim of this section is the following result:

\begin{prop}\label{prop:eps1}
    If $2a\kappa<\gamma$ and $f\in C^\infty(\M)$, then
    \[
        \limsup_{t\to\infty} \frac 1 t \log\left(\sup_{s\ge0} \abs{\int_t^{t+s}\eps^1_r(f)dr}\right) \le -\left(\frac{\gamma}{2}-a\kappa\right) \quad \text{almost surely}.
    \]
\end{prop}

While Proposition \ref{prop:eps1} is mostly just a reformulation of \cite[Theorem 2.1]{Raimond}, the exact value on the right hand side is not provided there. Because of this, and since many of the intermediate results and the objects involved will also be crucial in Section \ref{sec:eps2} below, we will now give a detailed summary of Section 2 from \cite{Raimond}, including a streamlined proof of Proposition \ref{prop:eps1}.

Let us fix $\mu\in\cM(\M)$ and consider the time-homogeneous stochastic differential equation with no self-interaction
\begin{equation*}
	dY_t = \sqrt{2} dW_t(Y_t)- \nabla V_{\mu}(Y_t)dt,
\end{equation*}
where we think of the dynamics as those arising from freezing the drift potential in \eqref{eq:selfinteractingM} at some time $t_0$, so that formally $\mu=\beta(t_0)\mu_{t_0}$. These dynamics can also be described via the infinitesimal generator
\begin{equation}\label{eq:At}
    A_\mu f = \Delta f-\nabla V_\mu \cdot\nabla f \quad \text{for all }f\in \cD(A_\mu)\supset C^2(\M),
\end{equation}
and the corresponding equilibrium is given by $\Pi(\mu)$ as defined in \eqref{eq:PIfunctionalmeasure}. Let $(P_s^{\mu})_{s\ge0}$ denote the transition semigroup on $\LL^2(\Pi(\mu))=\LL^2(dx)$ generated by $A_\mu$. Then $A_\mu$ satisfies a Poincar\'e inequality and therefore $P_s^{\mu}f$ converges to $\Pi(\mu)(f)$ exponentially fast with respect to the $\LL^2$-norm for any $f\in\LL^2(dx)$ (see Sections 1.1 and 1.4.3 of \cite{Wang}). Using this and basic semigroup theory, one can easily show that
\begin{equation*}
    Q_\mu f:=-\int_0^\infty P^{\mu}_s(f-\Pi(\mu)(f))ds \quad \text{for all }f\in\LL^2(dx),
\end{equation*}
is well-defined and satisfies
\begin{equation*}
    A_\mu Q_\mu f=Q_\mu A_\mu f=f-\Pi(\mu) (f) \quad \text{for all }f\in\cD(A_\mu),
\end{equation*}
so $Q_\mu$ is "almost an inverse to $A_\mu$".

Now set
\[
A_t:=A_{\beta(t)\mu_t}, \quad Q_t:=Q_{\beta(t)\mu_t}.
\]
Then for all $f\in C^\infty(\M)$ we can rewrite \eqref{eq:def-eps1} as
\[
    \eps^1_t(f)=f(X_{e^t})-\Pi(\beta(e^t)\mu_{e^t})(f)=A_{e^t}Q_{e^t}f(X_{e^t}).
\]
If we set
\begin{equation}\label{eq:Ft}
    F^f_t(x):=\frac{1}{t}Q_tf(x), 
\end{equation}
applying the change of variables $r\mapsto \log r$ and then Ito's formula yields
\begin{equation}\label{eq:epsIto}
        \int_s^t\eps^1_r(f)dr
        =\int_{e^s}^{e^t}A_rF^f_r(X_r)dr =F^f_{e^t}(X_{e^t})-F^f_{e^s}(X_{e^s})- \int_{e^s}^{e^t} \dot F^f_r (X_r)dr + M^{f,s}_t,
\end{equation}
where $(M^{f,s}_t)_{t\ge s}$ is a martingale with
\begin{equation}\label{eq:anglebracket}
     \langle M^{f,s},M^{g,s}\rangle_t= \int_{s}^{t} e^r \nabla F^f_{e^r}(X_{e^r})\cdot \nabla F^g_{e^r}(X_{e^r}) dr \quad \text{for all $f,g\in C^\infty(\M)$}
\end{equation}
(compare \cite[Section 3.3]{Raimond} and \cite[Remark 2.2]{sidI}). Therefore, in order to estimate the integral over $\eps^1$, we need to estimate $F^f_t$, $\nabla F^f_t$, and the time derivative $\dot F^f_t$, which can be done with the help of the following lemma.

\begin{lem}\label{lem:Qt}
There is a constant
\begin{equation}\label{eq:kappa}
    \kappa \in(0,\infty)
\end{equation}
depending only on the dimension of $\M$ such that for all $f\in C^\infty(\M)$ and $t\ge t_0$ the following estimates hold.
\begin{enumerate}
    \item
    $\norm{Q_tf}_\infty \le C t^{a \kappa}\norm{f}_\infty$,
    \item
    $\norm{\nabla Q_tf}_\infty \le C (1+\log t)^{1/2}t^{a \kappa}\norm{f}_\infty$,
    \item
    $\norm{\partial_t Q_t f}_\infty \le C (\log t)^{3/2} t^{2a\kappa-\gamma} \norm{f}_\infty$.
\end{enumerate}
\end{lem}

\begin{proof}
We can take $\kappa$ as the constant of the same name from \cite[Lemma 2.3]{Raimond}, which then yields
\[
    \norm{Q_tf}_\infty
    =\norm{Q_{\beta(t)\mu_t}f}_\infty
    \le C e^{\kappa\norm{V_{\beta(t)\mu_t}}_\infty}\norm{f}_\infty.
\]
With Assumptions \ref{ass:V} and \ref{ass:beta}, we get
\[
    \norm{V_{\beta(t)\mu_t}}_\infty
    \le \abs{\beta(t)} \abs{\mu_t(v)} \norm{v}_\infty
    \le a \log t,
\]
and the first estimate follows. As shown in the proof of \cite[Lemma 2.3]{Raimond}, $\kappa$ depends only on the dimension of $\M$. The second estimate also follows from \cite[Lemma 2.3]{Raimond} in a similar way, and the third estimate can be quoted straight from \cite[Lemma 2.8]{Raimond}.
\end{proof}

\begin{remark}\label{rem:kappa}
From now on, $\kappa$ will always be the constant from \eqref{eq:kappa}. In particular, this is the same $\kappa$ that we use in Theorem \ref{thm:main}. As seen in the proof of Lemma \ref{lem:Qt} above, its origin lies in \cite[Lemmas 2.3]{Raimond} which is proved via classical results about $\log$-Sobolev and Poincar{\'e} inequalities for the operator $A_t$ and an application of the Bakry-Emery criterion. The constant $\kappa$ is derived from multiple different constants occurring in the process.  
\end{remark}

Again, consider $F^{f}_t$ as in \eqref{eq:Ft} with $f\in C^\infty(\M)$. As mentioned above, Lemma \ref{lem:Qt} is the key to
estimating the terms $F^f_t$, $\nabla F^f_t$, and $\dot F^f_t$ that occur in \eqref{eq:epsIto} and \eqref{eq:anglebracket}. Indeed, the first two estimates from Lemma \ref{lem:Qt} imply
\begin{equation}\label{eq:lemma-F}
    \norm{F^f_{e^t}}_\infty \le Ce^{-(1-a\kappa)t}\norm{f}_\infty
\end{equation}
and
\begin{equation}\label{eq:lemma-gradF}
    \norm{\nabla F^f_{e^t}}_\infty \le C(1+t)^{1/2} e^{-(1-a\kappa)t}\norm{f}_\infty.
\end{equation}
Noting that
\begin{equation*}
    \dot F^f_t = -\frac{1}{t^2} Q_tf + \frac{1}{t} \d_t Q_tf,
\end{equation*}
the first and third estimates from Lemma \ref{lem:Qt} imply
\begin{equation}\label{eq:lemma-Fdot}
    \norm{\dot F^f_{e^t}}_\infty
    \le C \left(e^{(a\kappa-2)t} + t^{3/2} e^{(2a\kappa-\gamma-1)t}\right)\norm{f}_\infty
    \le C (1+t)^{3/2} e^{(2a\kappa-\gamma-1)t}\norm{f}_\infty,
\end{equation}
where the last step used that $\gamma\in(0,1]$ by Assumption \ref{ass:beta}.

\begin{proof}[Proof of Proposition \ref{prop:eps1}]
Let $t\ge t_0$, $s\ge 0$, and let $F_t=F^{f}_t$ with $f\in C^\infty(\M)$ as in \eqref{eq:Ft}. Writing $M=M^{f,t}$ for short, we get from \eqref{eq:epsIto} that
\begin{equation}\label{eq:eps1proof}
    \int_t^{t+s}\eps^1_r(f)dr
    = F_{e^{t+s}}(X_{e^{t+s}}) - F_{e^t}(X_{e^{t}})
      + \int_t^{t+s} e^r\dot F_{e^r} (X_{e^r}) dr
      + M_{t+s},
\end{equation}
and we will now estimate all of the terms on the right hand side. Let us write
\[
    \xi=\frac{\gamma}{2}-a\kappa
\]
for short. By \eqref{eq:lemma-F}, we get
\begin{equation}\label{eq:F-bound}
    \abs{F_{e^{t+s}}(X_{e^{t+s}}) - F_{e^t}(X_{e^{t}})}
    \le C e^{-(1-a\kappa)t}\norm{f}_\infty
    \le C e^{-\xi t}\norm{f}_\infty.
\end{equation}
By \eqref{eq:lemma-Fdot},
\begin{equation}\label{eq:F-dot-bound}
    \abs{\int_t^{t+s} e^r\dot F_{e^r} (X_{e^r}) dr}
    \le C\norm{f}_\infty\int_t^{t+s} (1+r)^{3/2} e^{-2\xi r} dr 
    \le C\norm{f}_\infty\int_t^{\infty} e^{-\xi r} dr
    \le C e^{-\xi t} \norm{f}_\infty.
\end{equation}
Now, let $\delta>0$. With the help of \eqref{eq:anglebracket} and \eqref{eq:lemma-gradF}, we get
\begin{equation}\label{eq:anglebracketbound}
    \langle M\rangle_{t+s}
    = \int_t^{t+s} e^r \abs{\nabla F_{e^r}(X_{e^r})}^2dr
    \le C \norm{f}_\infty\int_t^\infty(1+r) e^{-(1-2a\kappa)t}dr 
    \le C e^{-(2\xi-\delta)t} \norm{f}_\infty.
\end{equation}
Let $\eps>0$. Markov's inequality, Burkholder's inequality and \eqref{eq:anglebracketbound} imply
\[
    \P\left[ e^{(\xi-\delta)n} \sup_{s\ge 0} \abs{M_{n+s}}\ge \eps \right]
    \le \eps^{-2} e^{2(\xi-\delta)n} \E\left[\sup_{s\ge 0} \abs{M_{n+s}}^2\right]
    \le C \eps^{-2} e^{-\delta n} \norm{f}_\infty  \quad \text{for all $n\in\N$}.
\]
Since the right hand side is summable in $n\in\N$, Borel-Cantelli implies
\[
    e^{(\xi -\delta)n} \sup_{s\ge0} \abs{M_{n+s}} \xrightarrow{n\to\infty}0 \quad \text{almost surely for any $\delta>0$.}
\]
Combining this with \eqref{eq:eps1proof}, \eqref{eq:F-bound} and \eqref{eq:F-dot-bound}, we arrive at
\[
    \limsup_{n\to\infty} \frac 1 n \log \left( \sup_{s\ge0} \abs{\int_n^{n+s}\eps^1_r(f)dr}\right) \le -\xi \quad \text{almost surely.}
\]
If $\lfloor t \rfloor$ denotes the integer part of $t\in(0,\infty)$, then
\[
    \sup_{s\ge0} \abs{\int_t^{t+s}\eps^1_r(f)}dr = \sup_{s\ge0} \abs{\left(\int_{\lfloor t \rfloor}^{t+s}-\int_{\lfloor t \rfloor}^t\right)\eps^1_r(f)dr} \le 2 \sup_{s\ge0} \abs{\int_{\lfloor t \rfloor}^{{\lfloor t \rfloor}+s}\eps^1_r(f)dr},
\]
so we also get the continuous-time property as stated in the proposition.
\end{proof}


\subsection{Dealing with $\eps^2$}\label{sec:eps2}

The aim of this section is to find a suitable counterpart to Proposition \ref{prop:eps1} for $\eps^2$. We start with some preparatory lemmas. Denote the total variation norm of $\mu\in\cM(\M)$ by
\begin{equation*}
    \norm{\mu}=\sup\{ \mu(f) : f\in C(\M),\, \norm{f}_\infty\le 1\},
\end{equation*}
and recall that
\begin{equation}\label{eq:PI(m)}
    \Pi(m)(dx)=\frac{e^{-m\cdot v(x)}}{Z(m)}dx \quad \text{with} \quad Z(m)=\int_\M e^{-m\cdot v(x)}dx
\end{equation}
for all $m\in\R^N$.

\begin{lem}\label{lem:Pi}
The mapping $\Pi\colon\R^N\to\cM(\M)$ is Lipschitz continuous, i.e.
\begin{equation*}
    \norm{\Pi(m)-\Pi(m')} \le C \abs{m-m'} \quad \text{for all $m,m'\in\R^N$.}
\end{equation*}
\end{lem}

\begin{proof}
It is easy to check that the derivative of $\Pi$ is bounded.
\end{proof}

\begin{lem}\label{lem:Pi-Z}
For all $m\in\R^N$, we have
\begin{equation}\label{eq:Pigradient}
    \Pi(m)(v)=-\nabla\log Z(m)
\end{equation}
and
\begin{equation}\label{eq:PiCov}
    \Cov_{\Pi(m)}(v)=\Hess \log Z(m),
\end{equation}
where $\Hess$ denotes the Hessian matrix.
\end{lem}

\begin{proof}
This is a straight forward calculation.
\end{proof}

\begin{lem}\label{lem:gradient}
If we set
\begin{equation*}
    J_t\colon \R^N\to\R,\quad m\mapsto\frac12\abs{m}^2+1_{\beta(e^t)\neq 0}\cdot \frac{1}{\beta(e^t)}\log Z(\beta(e^t)m),
\end{equation*}
for all $t\ge0$, then
\begin{equation}\label{eq:gradientJt}
    \nabla J_t (m)=m-\Pi(\beta(e^t)m)(v)  \quad \text{for all $m\in\R^N$ and $t\ge 0$}
\end{equation}
and
\begin{equation}\label{eq:convex}
    m \cdot \nabla J_t (m) \ge (1-\Lambda \beta_0) \abs{m}^2 \quad \text{for all $m\in\R^N$ and $t\ge t_0$}
\end{equation}
where $\Lambda$ and $\beta_0$ are the constants from Assumption \ref{ass:betaLambda}.
\end{lem}

\begin{proof}
Due to \eqref{eq:PI(m)} and \eqref{eq:assvcentered}, we have $\Pi(0)(v)=0$. Using this fact and \eqref{eq:Pigradient}, we can easily check that \eqref{eq:gradientJt} holds. Because of \eqref{eq:PiCov}, we get
\begin{align*}
    \Hess J_t(m)=1_{N\times N}+\beta(e^{t})\Cov_{\Pi(\beta(e^{t})m)}(v),
\end{align*}
and thanks to Assumption \ref{ass:betaLambda}, this implies that the smallest eigenvalue of $\Hess J_t(m)$ is at least $1-\Lambda \beta_0>0$. Therefore $J_t$ is strongly convex, more precisely
\begin{equation*}\label{eq:convex_pre}
    (m-m') \cdot \left(\nabla J_t (m)-\nabla J_t (m')\right) \ge (1-\Lambda \beta_0) \abs{m-m'}^2 \quad \text{for all $m,m'\in\R^N$.}
\end{equation*}
Since $\nabla J_t (0)=\Pi(0)(v)=0$, setting $m'=0$ yields \eqref{eq:convex}.
\end{proof}

\begin{remark}
$J_t(m)$, $\nabla J_t (m)$, and $\Hess J_t(m)$ are all continuous with respect to $t\in [0,\infty)$.
\end{remark}

The following lemma contains the essential step towards the main result of this section, Proposition \ref{prop:eps2}. In the sequel, we use the shorthand notation
\begin{equation}\label{eq:def-mt}
    m_t:=\mu_{e^t}(v)\in\R^N \quad \text{for all $t\ge0$,}    
\end{equation}
and $\eta$ will always be the constant from Theorem \ref{thm:main}, i.e.
\begin{equation*}
    \eta=\min\left\{\frac{\gamma}{2}-a\kappa, 1-\Lambda\beta_0\right\}>0.
\end{equation*}

\begin{lem}\label{prop:meantozero}
	If $2a\kappa<\gamma$, then
	\begin{equation*}
        e^{\delta t} m_t  \xrightarrow{t\to\infty} 0 \quad \text{almost surely for all $\delta<\eta$.}
	\end{equation*}
\end{lem}

\begin{proof}
    1.) Set
    \begin{equation}\label{eq:mtilde}
        \tilde m_t:=m_t-F_{e^t}(X_{e^t})
    \end{equation}
    where
    \begin{equation*}
        F_t(x):=(F_t^{v_i}(x))_{i=1,\ldots,N}=
            t^{-1}(Q_tv_i(x))_{i=1,\ldots,N}
    \end{equation*}
    (compare \eqref{eq:Ft}). Let $t>t_0$. By \eqref{eq:mtilde} and \eqref{eq:lemma-F}, we have
    \begin{equation*}
        |e^{\delta t} m_t|
        \le e^{\delta t} |\tilde m_t| + Ce^{-(1-a\kappa-\delta)t}.
    \end{equation*}
    Since $\delta<\frac{\gamma}{2}-a\kappa<1-a\kappa$, the second summand in the above bound vanishes for $t\to\infty$, and hence it suffices to show for this lemma that
    \begin{equation}\label{eq:tildeconvergence}
        e^{\delta t} \tilde m_t  \xrightarrow{t\to\infty} 0 \quad \text{almost surely for all $\delta<\eta$.}
	\end{equation}
	Also note that for all $t>t_0$, both $m_t$ and $\tilde m_t$ have deterministic bounds, as
    \[
        \abs{m_t}=\abs{\mu_{e^t}(v)}\le \norm{v}_\infty =1   
    \] 
    and hence, by \eqref{eq:mtilde} and \eqref{eq:lemma-F},
    \begin{equation}\label{eq:mtildebounded}
        \abs{ \tilde m_t}\le 1+ Ce^{-(1-a\kappa)t}\le C.   
    \end{equation}
    
    2.) In this step, we use a similar approach as in Sections 3.1 - 3.4 of \cite{Raimond} in order to find a stochastic differential equation that is fulfilled by $|\tilde m_t|^2$. Combining \eqref{eq:ODE+eps1} and \eqref{eq:gradientJt} (from Lemma \ref{lem:gradient}), we get 
    \begin{equation*}
        \dot m_t=-\nabla J_t (m_t)+ \eps^1_t(v),
    \end{equation*}
    and hence, by \eqref{eq:mtilde},
    \begin{equation*}
        d\tilde m_t=\left( -\nabla J_t (m_t) + \eps^1_t(v)\right)dt- d F_{e^t}(X_{e^t}).
    \end{equation*}
    By virtue of \eqref{eq:epsIto}, this can be rewritten as
    \begin{equation}\label{eq:increment-n}
        d\tilde m_t= \left(-\nabla J_t(m_t) - e^t\dot F_{e^t}(X_{e^t})\right)dt + dM^{v,t_0}_t,
    \end{equation}    
    where we read $M^{v,t_0}_t$ as the vector $(M^{v_i,t_0}_t)_{i=1,\ldots,N}$, and we will use the shorthand notations $M_t=M^{v,t_0}_t$ and $M^i_t=M^{v_i,t_0}_t$. Next, we set
    \begin{equation}\label{eq:H}
        H_t=-e^t \dot F_{e^t}(X_{e^t})- \left(\nabla J_t(m_t)-\nabla J_t(\tilde m_t) \right)
    \end{equation}
    and rewrite \eqref{eq:increment-n} as
    \begin{equation}\label{eq:SDEmtilde}
        d\tilde m_t= \big(-\nabla J_t(\tilde m_t) + H_t\big)dt +dM_t.
    \end{equation}
    Ito's formula yields
    \begin{equation*}
        d|\tilde m_t|^2
        = 2 \tilde m_t \cdot d \tilde m_t + \sum_{i=1}^N d \langle M^i\rangle_t,
    \end{equation*}
    and by plugging in \eqref{eq:SDEmtilde} and the expression for $\langle M^i\rangle_t$ from \eqref{eq:anglebracket}, we get
    \begin{equation}\label{eq:SDE-n^2}
        d|\tilde m_t|^2
        = \left(-2 \tilde m_t \cdot \nabla J_t(\tilde m_t) +2 \tilde m_t\cdot H_t +e^t\sum_{i=1}^N \abs{\nabla F_{e^t}^{v_i}(X_{e^t})}^2\right)dt+2 \tilde m_t\cdot dM_t.         
    \end{equation}
    
    3.) Our next goal is to find a suitable upper bound for the drift in \eqref{eq:SDE-n^2}. For the rest of the proof we assume that $t\ge t_0$ and 
    \[
        0<\alpha< 2\eta=\min\left\{ \gamma-2a\kappa, 2(1-\Lambda \beta_0)\right\}.
    \]
    By \eqref{eq:lemma-Fdot},
    \begin{equation}\label{eq:Hbound1}
        \abs{e^t \dot F_{e^t}(X_{e^t})}
        \le C (1+t)^{3/2} e^{(2a\kappa-\gamma)t}
        \le C e^{-\alpha t}.
    \end{equation}
    By \eqref{eq:gradientJt} and Lemma \ref{lem:Pi}, we get
    \begin{align*}
        \abs{\nabla J_t(m_t)-\nabla J_t(\tilde m_t)}
        &\le \abs{m_t-\tilde m_t}+ \abs{\Pi(\beta(e^t)m_t)(v) - \Pi(\beta(e^t)\tilde m_t)(v)} \\
        &\le (1+C|\beta(e^t)|)\abs{m_t-\tilde m_t} \\
        &= (1+C|\beta(e^t)|)\abs{F_{e^t}(X_{e^t})}.
    \end{align*}
    and, with the help of \eqref{eq:lemma-F} and the logarithmic bound for $\abs{\beta}$ from Assumption \ref{ass:beta}, we arrive at
    \begin{equation}\label{eq:Hbound2}
        \abs{\nabla J_t(m_t)-\nabla J_t(\tilde m_t)}
        \le  C (1+t) e^{(a\kappa-1)t}
        \le C e^{-\alpha t}.
    \end{equation}
    Combining \eqref{eq:Hbound1}, \eqref{eq:Hbound2}, and \eqref{eq:mtildebounded}, we can conclude that with $H$ from \eqref{eq:H} we have
    \begin{equation}\label{eq:Htermdecay}
        \abs{\tilde m_t\cdot H_t} \le C e^{-\alpha t}.
    \end{equation}
    For the third summand of the drift term in \eqref{eq:SDE-n^2}, we apply \eqref{eq:lemma-gradF} and get
    \begin{equation}\label{eq:gradientdecay}
        e^t\sum_{i=1}^N \abs{\nabla F_{e^t}^{v_i}(X_{e^t})}^2
        \le C (1+t)e^{(2a\kappa-1)t}
        \le C e^{-\alpha t}.
    \end{equation}
    Finally, plugging \eqref{eq:Htermdecay} and \eqref{eq:gradientdecay} as well as \eqref{eq:convex} from Lemma \ref{lem:gradient} into \eqref{eq:SDE-n^2} yields
	\begin{equation}\label{eq:SDI-n}
        d|\tilde m_t|^2\le \left( -2 (1-\Lambda \beta_0)|\tilde m_t|^2 + C e^{-\alpha t}\right)dt+2 \tilde m_t \cdot dM_t.
    \end{equation}
    
	4.) With \eqref{eq:SDI-n} at hand, we can now investigate the asymptotics of $\tilde m_t$. Setting
	\[
	    \xi:=2(1-\Lambda \beta_0)
	\]
	and plugging \eqref{eq:SDI-n} into
    \[
        d\left( e^{\xi t}|\tilde m_t|^2 \right)=\xi e^{\xi t}|\tilde m_t|^2dt+ e^{\xi t}d|\tilde m_t|^2,
    \]
	we get
	\begin{align}\label{eq:step4}
        |\tilde m_t|^2
        &\le e^{-\xi (t-t_0)} |\tilde m_{t_0}|^2 + C\int_{t_0}^t e^{-\xi (t-r)-\alpha r} dr + 2\int_{t_0}^t e^{-\xi (t-r)}\tilde m_r \cdot  dM_r
        \le Ce^{-\alpha t} + 2e^{-\xi t}N_t,
    \end{align}
    where
    \[
        N_t=\int_{t_0}^t e^{\xi r}\tilde m_r\cdot dM_r.
    \]
    By \eqref{eq:anglebracket} and \eqref{eq:gradientdecay}, we have
    \begin{equation}\label{eq:angleestimate1}
        \langle N\rangle_t 
        \le \int_{t_0}^t e^{2\xi r} \abs{\tilde m_r}^2 e^{r} \abs{\nabla F_{e^r}(X_{e^r})}^2 dr \le C \int_{t_0}^t e^{(2\xi-\alpha)r} \abs{\tilde m_r}^2 dr.
    \end{equation}
    Since $\tilde m$ is bounded (see \eqref{eq:mtildebounded}) and $\alpha<2\xi$, \eqref{eq:angleestimate1} yields
	\begin{equation}\label{eq:angleestimate2}
        \langle N\rangle_t \le Ce^{(2\xi-\alpha) t},
	\end{equation}
	so the law of the iterated logarithm implies
	\begin{equation}\label{eq:step4bpre}
        \limsup_{t\to\infty} \frac{\abs{N_t}}{e^{\left(\xi -\frac{\alpha}{2}\right)t}\log(t)} <\infty \quad \text{almost surely.}
	\end{equation}
    In analogy to the way we treat constants $C$, we will now write $K$ for an almost surely finite non-negative random variable, that may change from one step to the next with no indication. Using this notational convention, \eqref{eq:step4bpre} is equivalent to
    \begin{equation}\label{eq:step4b}
        \abs{N_t}\le K e^{\left(\xi -\frac{\alpha}{2}\right)t}\log(t) \quad \text{for all $t>t_0$}.
	\end{equation}
    Plugging \eqref{eq:step4b} into \eqref{eq:step4} (and enlarging $t_0$, if necessary) yields
    \begin{align*}
        |\tilde m_t|
        \le \left( Ce^{-\alpha t} + Ke^{-\frac{\alpha}{2} t}\log(t) \right)^{\frac 12}
        \le K e^{-\frac{\alpha}{4} t}(\log(t))^{\frac 12}
        \quad \text{for all $t>t_0$},
    \end{align*}
    and hence
	\begin{equation}\label{eq:step4c}
         e^{\delta t} \tilde m_t  \xrightarrow{t\to\infty} 0 \quad \text{almost surely for all $\delta<\frac\alpha4$.}
	\end{equation}
    In particular, if $\delta<\frac\alpha4$, then $e^{\delta t} \tilde m_t$ is almost surely bounded. We can use this fact in \eqref{eq:angleestimate1} in order to improve \eqref{eq:angleestimate2} to
    \begin{equation*}
        \langle N\rangle_t \le K e^{(2\xi-(\alpha+2\delta)) t}
	\end{equation*}
	for any $\delta<\frac{\alpha}{4}$. Then we use the same argument as from \eqref{eq:angleestimate2} to \eqref{eq:step4c} in order to improve \eqref{eq:step4c} to
	\begin{equation*}
        e^{\delta t} \tilde m_t  \xrightarrow{t\to\infty} 0 \quad \text{almost surely for all $\delta<\frac{3\alpha}{8}$.}
	\end{equation*}
	Iterating this argument shows that for any $n\in\N$ we have
	\begin{equation*}
        e^{\delta t} \tilde m_t  \xrightarrow{t\to\infty} 0 \quad \text{almost surely for all $\delta<\frac{(2^n-1)\alpha}{2^{n+1}}$.}
	\end{equation*}
	Since $n$ can be arbitrarily large and $\alpha$ arbitrarily close to $2\eta$, it follows that \eqref{eq:tildeconvergence} holds, and thus the proof is completed.
\end{proof}

\begin{remark}
    Note that Lemma \ref{prop:meantozero} can be interpreted as a statement on the polynomial decay of the drift potential $\beta(t) V_{\mu_t}$ of \eqref{eq:selfinteractingM}. More precisely, since $\norm{V_{\mu_t}}_\infty \le \abs{m_{\log t}}$ (by \eqref{eq:V_t} and \eqref{eq:def-mt}), Lemma \ref{prop:meantozero} and Assumption \ref{ass:beta} imply
	\[
        \limsup_{t\to\infty} \frac{\log \norm{\beta(t)V_{\mu_t}}_\infty}{\log t} \le -\eta \quad \text{almost surely},
    \]
    provided that $2a\kappa<\gamma$.
\end{remark}

We can now prove the main result of this section. 

\begin{prop}\label{prop:eps2}
    If $2a\kappa<\gamma$, then almost surely
    \[
            \limsup_{t\to\infty} \frac 1 t \log \norm{\eps^2_t} \le -\eta.
    \]
\end{prop}

\begin{proof}
By \eqref{eq:def-eps2} and \eqref{eq:PI(m)},
\begin{equation*}
    \eps^2_t= \Pi(\beta(e^t)\mu_{e^t}(v))-\Pi(0),
\end{equation*}
so Lipschitz continuity (see Lemma \ref{lem:Pi}) and the logarithmic bound for $\abs{\beta}$ from Assumption \ref{ass:beta} imply
\begin{equation*}
    \norm{\eps^2_t}\le C t \abs{\mu_{e^t}(v)} = C t \abs{m_t}.
\end{equation*}
The claim now follows from Lemma \ref{prop:meantozero}.
\end{proof}


\subsection{Putting the pieces together}\label{sec:mainproof}

Let $f\in C(\M)$. The unique solution to \eqref{eq:limitODEnewreal} at time $t\ge 0$ with the initial value $x\in\R$ is given by
\[
    \Phi_f(t,x)=e^{-t}x+(1-e^{-t})\cU(f),
\]
so the mapping $\Phi_f\colon [0,\infty)\times\R\to\R$ is the semiflow on $\R$ that is generated by \eqref{eq:limitODEnewreal}. Recall again the constant
\begin{equation*}
    \eta=\min\left\{\frac{\gamma}{2}-a\kappa, 1-\Lambda\beta_0\right\}>0
\end{equation*}
from Theorem \ref{thm:main}.

\begin{prop}\label{prop:apt}
    If $2a\kappa<\gamma$ and $f\in C^\infty(\M)$, then almost surely $(\mu_{e^t}(f))_{t\ge0}$ is a $(-\eta)$-pseudotrajectory of the semiflow $\Phi_f$, i.e. almost surely
    \begin{equation}\label{eq:apt}
        \limsup_{t\to\infty} \frac 1 t \log\left(\sup_{s\in[0,T]} \abs{ \mu_{e^{t+s}}(f)-\Phi_f(s,\mu_{e^t}(f))}\right) \le -\eta \quad \text{for all $T>0$}.
    \end{equation}
\end{prop}

\begin{proof}
    By \eqref{eq:ODE+eps} and \eqref{eq:limitODEnewreal} we have
    \begin{equation*}
        \mu_{e^{t+s}}(f) - \Phi_f(s,\mu_{e^t}(f))
        = -\int_0^s \left(\mu_{e^{t+r}}(f) - \Phi_f(r,\mu_{e^t}(f))\right)dr + \int_0^s\left(\eps^1_{t+r}+\eps^2_{t+r}\right)(f)dr,
    \end{equation*}
    so, by variation of constants,
    \begin{equation}\label{eq:aptevosolution}
        \mu_{e^{t+s}}(f) - \Phi_f(s,\mu_{e^t}(f))
        = \int_t^{t+s}e^{-(t+s-r)}\left(\eps^1_r+\eps^2_r\right)(f)dr.
    \end{equation}
    Since integration by parts yields
    \[
        \int_t^{t+s}e^{-(t+s-r)} \eps^1_r(f) dr = \int_t^{t+s} \eps^1_r(f) dr - \int_t^{t+s} e^{-(t+s-r)}\left(\int_t^r\eps^1_u(f)du\right) dr,
    \]
    the claim follows from \eqref{eq:aptevosolution} and Propositions \ref{prop:eps1} and \ref{prop:eps2}.
\end{proof}

Intuitively, Proposition \ref{prop:apt} says that $\mu_{e^t}(f)$ is exponentially close to the behavior of a solution of \eqref{eq:limitODEnewreal}. Since any solution of \eqref{eq:limitODEnewreal} converges exponentially fast to $\cU(f)$, Theorem \ref{thm:main} now follows from a general result about pseudotrajectories.

\begin{proof}[Proof of Theorem \ref{thm:main}]
    Clearly,
    \begin{equation*}
    \limsup_{t\to\infty} \frac 1 t \log\abs{\Phi_f(t,x)-\cU(f)} = -1 \quad \text{for all $x\in\R$.}
    \end{equation*}
    Thanks to this and Proposition \ref{prop:apt}, all of the conditions of part (i) of \cite[Lemma 8.7]{Benaim1999} (with $B=\R$, $K=\{\cU(f)\}$, $X(t)=\mu_{e^t}(f)$, $Y(t)=\cU(f)$, $\alpha=-\eta >-1=\lambda$) are fulfilled, and hence we get
    \begin{equation}\label{eq:ergodic-d-eps}
        \limsup_{t\to\infty} \frac 1 t \log \abs{\mu_{e^t}(f)-\cU(f)} \le -\eta \quad \text{almost surely,}
    \end{equation}
    which is equivalent to \eqref{eq:thm}.
\end{proof}

\begin{remark}
Note that the final step of the proof of Proposition \ref{prop:apt} also works with the supremum in \eqref{eq:apt} being taken over all $s\ge0$ instead of only $s\in[0,T]$. However, this stronger variant of the pseudotrajectory property is not needed for the argument in the proof of Theorem \ref{thm:main} to work.
\end{remark}

\begin{remark}
One can also use a more technical variant of the argument given in this section to get a slightly different version of Theorem \ref{thm:main} that provides a convergence rate with respect to a family of random metrics on $\cM(\M)$. If $(f_n)_{n\in\N} \subset C^\infty(\M)$ is dense in the unit sphere of $C(\M)$ and $c=(c_n)_{n\in\N}\subset (0,\infty)$ is summable, then
\begin{equation*}
    d(\mu,\nu)
    =\sum_{n=1}^\infty c_n\abs{\mu(f_n)-\nu(f_n)}
    \quad \text{for all $\mu,\nu\in\cM(\M)$}
\end{equation*}
is a metric on $\cM(\M)$ that induces the weak convergence of measures, i.e. for $\mu,\mu_1,\mu_2,\ldots\in \cM(\M)$ we have
\begin{equation*}
    d_c(\mu_k,\mu)\xrightarrow{k\to\infty} 0
    \quad \iff \quad
    \mu_k(f)\xrightarrow{k\to\infty} \mu(f) \quad \text{for all $f\in C(\M)$.}
\end{equation*}
By Proposition \ref{prop:eps1}, there exist finite random variables $K_n(\eps)$ such that almost surely
\[
    \sup_{s\ge0} \abs{\int_t^{t+s}\eps^1_r(f_n)dr} \le K_n(\eps) e^{-\left(\frac{\gamma}{2}-a\kappa-\eps\right)t} \quad \text{for all $t\ge t_0$}.
\]
With the random sequence $c(\eps)=(c_n(\eps))_{n\in\N}$ defined by $c_n(\eps)=(K_n(\eps))^{-1}2^{-n}$, the same argument as in the proof of Proposition \ref{prop:apt} yields that almost surely $(\mu_{e^t})_{t\ge0}$ is a $(-\eta+\eps)$-pseudotrajectory of the semiflow
\[
    \Phi\colon [0,\infty)\times\cM(\M)\to\cM(\M),\quad (t,\nu_0)\mapsto e^{-t}\nu_0+(1-e^{-t})\cU,
\]
on $(\cM(\M),d_{c(\eps)})$. Ultimately, this leads to the conclusion that under the same assumptions as in Theorem \ref{thm:main} we have
\begin{equation*}
    \limsup_{t\to\infty} \frac{\log d_{c(\eps)}(\mu_t,\cU)}{\log t}
    \le -(\eta-\eps)
    \quad \text{almost surely for all $\eps>0$.}
\end{equation*}
\end{remark}

\section{A closer investigation of the case $\M=\S^n$ and $v(x)=x$}\label{sec:sphere}

For this entire section, let
\[
    \M=\S^n \quad \text{and} \quad v(x)=x.
\]
In the case of weak self-interaction, i.e.\ $\beta\equiv b$, \cite[Theorem 4.5]{sidI} then implies that\footnote{Note that our notation differs slightly from that in \cite{sidI}: the parameter $a$ there corresponds to what is $b/2$ in our notation, and there is another factor 2 in \cite[(11)]{sidI} that is not in our definition of $\Pi$ (see Remark \ref{rem:sqrt2}).}
\begin{equation}\label{eq:sidI-thm}
    \mu_t \xrightarrow{w} \cU \quad \text{almost surely}\quad \Leftrightarrow \quad b\ge -(n+1),
\end{equation}
while Theorem \ref{thm:main} implies
\begin{equation}\label{eq:sidI-thm-comparison}
    b>-\frac{1}{\Lambda}
    \quad \Rightarrow \quad
    \limsup_{t\to\infty} \frac{\log \abs{(\mu_t-\cU)(f)}}{\log t} \le -\eta(b)<0 \quad \text{almost surely for all $f\in C^\infty(\S^n)$}
\end{equation}
for a suitable $\eta(b)>0$ (see Example \ref{ex:weak}). Because of \eqref{eq:sidI-thm} and \eqref{eq:sidI-thm-comparison}, we already know that $\Lambda\ge (n+1)^{-1}$. In this section, we provide a way to calculate $\Lambda$, show numerically that it is in fact strictly larger than $(n+1)^{-1}$, but then prove that the conclusion of \eqref{eq:sidI-thm-comparison} nevertheless actually holds for all $b> -(n+1)$.

\subsection{Calculating $\Lambda$}\label{sec:calculatingLambda}

In the following, for any $m\in\R^{n+1}$ we use the notation
\[
        \bar m =
        \begin{cases}
        \frac{m}{|m|}, & \text{if $m\neq 0$,}\\
        0, & \text{if $m= 0$}.
        \end{cases}
\]

\begin{lem}\label{lem:PiCov}
Let
\begin{align*}
    \rho(r)
    =-\frac{\int_0^\pi \cos x e^{-r\cos x} (\sin x)^{n-1}dx}{\int_0^\pi e^{-r\cos x} (\sin x)^{n-1}dx}
    \quad \text{for all $r\ge0$}.
\end{align*}
\begin{enumerate}
    \item
    For all $m\in\R^{n+1}$ we have    
    \begin{equation}\label{eq:expectation}
        \Pi(m)(v)= -\rho(|m|) \bar m.
    \end{equation}
    \item
    We have
    \[
        \lim_{r\to 0}\frac{\rho(r)}{r}=\frac{1}{n+1},
    \]
    and we therefore interpret $\frac{\rho(0)}{0}$ as $\frac{1}{n+1}$ in the following.
    \item
    For all $m\in\R^{n+1}$ we have
    \begin{equation*}
        \Cov_{\Pi(m)}(v)
        =\left(\rho'(|m|)-\frac{\rho(|m|)}{|m|} \right) \bar m \bar m^T  +          \frac{\rho(|m|)}{|m|} 1_{(n+1)\times(n+1)}   
    \end{equation*}
    and its largest eigenvalue is given by $\lambda(|m|)$, where
    \begin{equation}\label{eq:lambda}
        \lambda(r)=\frac{2\rho(r)}{r}-\rho'(r) \quad \text{for all $r\ge0$}.
    \end{equation}
    \item
    We have
    \begin{equation}\label{eq:lambdamax}
        \Lambda=\max_{r\ge 0} \lambda(r) \in \left[\frac{1}{n+1}, \frac{2}{n+1}\right).
    \end{equation}
\end{enumerate}
\end{lem}

\begin{figure}
    \centering
    \includegraphics[width=0.48\textwidth]{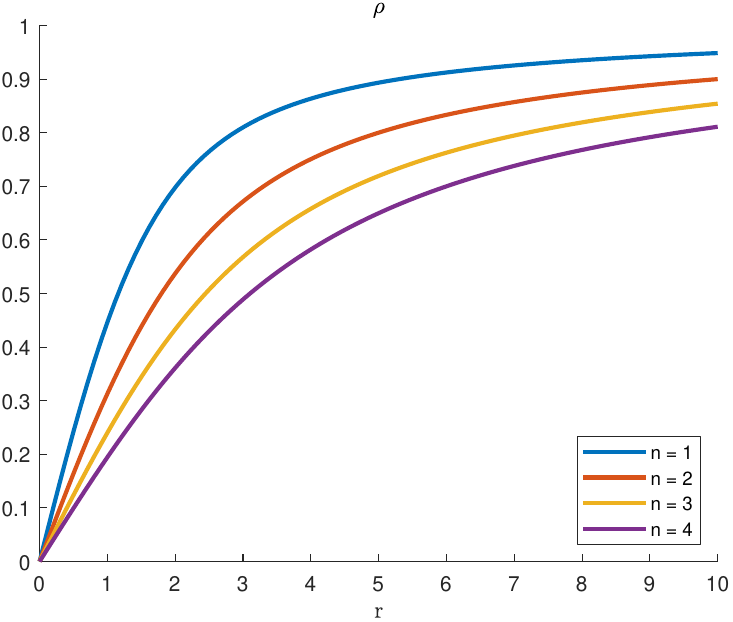}
    \includegraphics[width=0.48\textwidth]{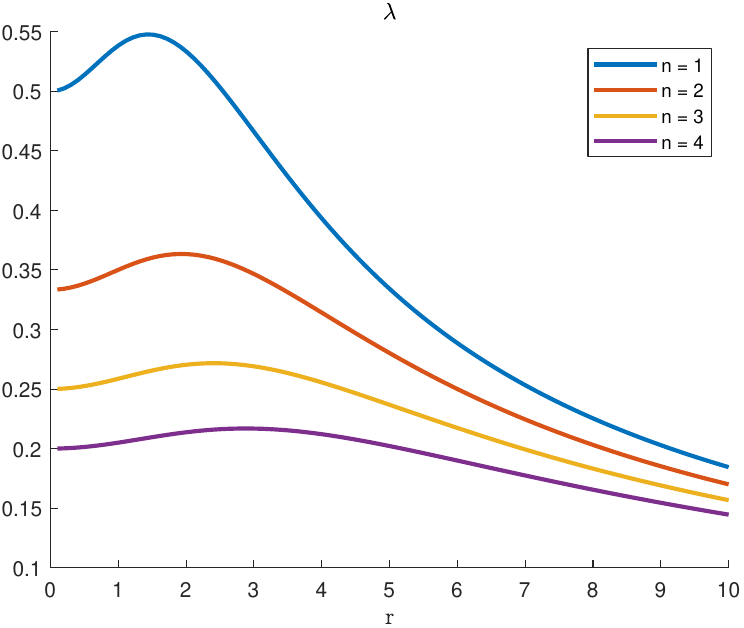}
    \caption{Plots of $\rho$ and $\lambda$ for $n\in\{1,2,3,4\}$.}
    \label{fig:rholambda}
\end{figure}

\begin{proof}
    The first part of this lemma is just a reformulation of \cite[Lemma 4.7]{sidI}. Note that for all $r\ge0$
    \begin{align*}
        \rho(r)=\frac{H'(r)}{H(r)} \quad \text{with} \quad  H(r)=\int_0^\pi e^{-r\cos x} (\sin x)^{n-1}dx
    \end{align*}
    and hence
    \begin{equation}\label{eq:rho'}
        \rho'(r)=\frac{H''(r)}{H(r)}-\left(\frac{H'(r)}{H(r)}\right)^2>0
    \end{equation}
    where the inequality follows from Cauchy-Schwarz. As shown in the proof of \cite[Lemma 4.8]{sidI},
    \begin{equation*}
        \frac{d}{dr}\frac{H''(r)}{H(r)}>0, \quad \frac{H''(0)}{H(0)}=\frac{1}{n+1}, 
    \end{equation*}
    and since an integration by parts yields
    \begin{equation}\label{eq:ODE-H}
        H'(r)=\frac{r}{n}(H(r)-H''(r)) \quad \text{for all $r\ge0$,}
    \end{equation}
    we get
    \begin{equation}\label{eq:monotony-rho/r}
        \frac{d}{dr}\frac{\rho(r)}{r}=\frac{d}{dr}\left(\frac{1}{n}\left(1-\frac{H''(r)}{H(r)}\right)\right)<0, \quad \frac{\rho(r)}{r}\xrightarrow{r\to0}\frac{1}{n+1}.
    \end{equation}
    In particular, we have
    \[
        0>\frac{d}{dr}\frac{\rho(r)}{r}=\frac{\rho'(r)}{r}-\frac{\rho(r)}{r^2} \quad \text{for all $r>0$}
    \]
    so, together with \eqref{eq:rho'},
    \begin{equation}\label{eq:rho'-bounds}
        0<\rho'(r)<\frac{\rho(r)}{r} \quad \text{for all $r>0$}.
    \end{equation}
    Combining \eqref{eq:expectation} with Lemma \ref{lem:Pi-Z}, we get
    \begin{align*}
        \Cov_{\Pi(m)}(v)
        &=\left( \partial_{m_i} \left(\frac{\rho(|m|)}{|m|} m_j\right)  \right)_{i,j=1,\ldots,n+1}=\left(\rho'(|m|)-\frac{\rho(|m|)}{|m|} \right)\bar m \bar m^T  +          \frac{\rho(|m|)}{|m|} 1_{(n+1)\times(n+1)}
    \end{align*}
    and hence its largest eigenvalue is
    \begin{align*}
        \sup_{x \in \S^n} x^T \Cov_{\Pi(m)}(v)x = \abs{\rho'(|m|)-\frac{\rho(|m|)}{|m|}} + \frac{\rho(|m|)}{|m|}=\lambda(|m|),
    \end{align*}
    where the last equality uses \eqref{eq:rho'-bounds}. Plugging \eqref{eq:rho'-bounds} into \eqref{eq:lambda} yields
    \begin{equation*}
        \frac{\rho(r)}{r}<\lambda(r)<\frac{2\rho(r)}{r} \quad \text{for all $r>0$,}
    \end{equation*}
    which, in combination with \eqref{eq:monotony-rho/r}, implies \eqref{eq:lambdamax}.
\end{proof}

\begin{remark}
    With the help of Lemma \ref{lem:PiCov}, one can easily show that if $(m_k)_{k\in\N}\subset\R^{n+1}$ satisfies $|m_k|\to\infty$ and $\bar m_k\to m\in\S^n$ for $k\to\infty$, then $\Pi(m_k)\xrightarrow{w}\delta_{-m}$ for $k\to\infty$.
\end{remark}

\begin{table}
    \centering
    \begin{tabular}{|c||c|c|c|} \hline
        $n$ & $\Lambda=\max \lambda $ & $\Lambda\cdot (n+1)$ & $\argmax \lambda$ \\ \hline\hline
        1 & 0.548 & 1.096 & 1.442 \\
        2 & 0.363 & 1.090 & 1.930 \\
        3 & 0.272 & 1.087 & 2.405 \\
        4 & 0.217 & 1.084 & 2.876 \\
        5 & 0.180 & 1.083 & 3.345 \\
        6 & 0.154 & 1.081 & 3.812 \\
        7 & 0.135 & 1.080 & 4.278 \\
        8 & 0.120 & 1.079 & 4.744 \\
        9 & 0.108 & 1.079 & 5.210 \\
        10 & 0.098 & 1.078 & 5.676 \\
        20 & 0.051 & 1.075 & 10.329 \\
        50 & 0.021 & 1.073 & 24.288 \\
        100 & 0.011 & 1.073 & 47.552 \\ \hline
    \end{tabular}
    \caption{Numerical approximations for $\Lambda$, $\Lambda\cdot (n+1)$, and $\argmax \lambda$.}
    \label{table:numerics}
\end{table}

\begin{remark}\label{rem:numerics}
    It follows from \eqref{eq:rho'} and \eqref{eq:ODE-H} that $\rho$ satisfies the differential equation
    \begin{equation*}
        \rho'(r)=1-\rho(r)\left(\frac{n}{r}+\rho(r)\right),
    \end{equation*}
    so thanks to \eqref{eq:lambda}, we can express both $\lambda(r)$ and $\lambda'(r)$ as functions of $r$ and $\rho(r)$. This makes it easy to calculate $\lambda$ and $\lambda'$ and thus approximate the maximum of $\lambda$ numerically. Simulations suggest that $\lambda(r)$ attains $\Lambda$ as its unique local and global maximum in a position $r=\argmax \lambda$ that grows linearly in $n$. Furthermore, these simulations suggest that $\Lambda \cdot (n+1)$ is indeed strictly greater than 1, even though it is decreasing in $n$ and the upper bound $2$ from \eqref{eq:lambdamax} is far from optimal. Still, even the suboptimal \eqref{eq:lambdamax} is a vast improvement over the trivial bound from \eqref{eq:Lambdatrivialbound}. See Table \ref{table:numerics} for some approximate values and Figure \ref{fig:rholambda} for a visualisation of $\rho$ and $\lambda$.
\end{remark}

\subsection{Improving Theorem \ref{thm:main} in the case of weak self-attraction}\label{sec:afterthought}

The following proposition means that we can close the gap between Theorem \ref{thm:main} (for $\M=\S^n$, $v(x)=x$ and constant $\beta$) and \cite[Theorem 4.5 (i)]{sidI} for all cases in which the uniform distribution is the limit.

\begin{prop}\label{prop:weaksphere}
    Let $\beta\equiv b>-(n+1)$ and $f\in C^\infty\left(\S^n\right)$. Then
    \begin{equation}\label{eq:thmweaksphere}
        \limsup_{t\to\infty} \frac{\log \abs{(\mu_t-\cU)(f)}}{\log t} \le -\min\left\{\frac12, 1+ \frac{b}{n+1}\right\} \quad \text{almost surely.}
    \end{equation}
\end{prop}

\begin{remark}
Note that we can technically also include the case $b=-(n+1)$ in Proposition \ref{prop:weaksphere}, but then its conclusion is not useful. Indeed, the right hand side of \eqref{eq:thmweaksphere} becomes 0 in this case, so this property is in fact weaker than what we already know from \eqref{eq:sidI-thm}.
\end{remark}

\begin{proof}[Proof of Proposition \ref{prop:weaksphere}]
    The case $b\ge0$ is already covered entirely in Example \ref{ex:weak}, so we assume that
    \[
        -(n+1)<b<0.
    \]
    Since Proposition \ref{prop:eps1} implies
    \[
        \limsup_{t\to\infty} \frac 1 t \log\left(\sup_{s\ge0} \abs{\int_t^{t+s}\eps^1_r(f)dr}\right) \le -\frac12 \quad \text{almost surely}
    \]
    (see Example \ref{ex:weak}), it only remains to show that
	\begin{equation}\label{eq:eps2weaksphere}
	        \limsup_{t\to\infty} \frac 1 t \log \norm{\eps^2_t} \le -\min\left\{\frac12, 1- \frac{|b|}{n+1}\right\} \quad \text{almost surely},
	\end{equation}
	since then the proof can be completed in the exact same way as in Section \ref{sec:mainproof}. In order to prove \eqref{eq:eps2weaksphere}, we will need to revisit and refine the arguments from Lemmas \ref{lem:gradient} and \ref{prop:meantozero} (the notation of which we adapt in the sequel) and make explicit use of \cite[Theorem 4.5 (i)]{sidI}.
	
    Since the second and third parts of Lemma \ref{lem:PiCov} imply
    \[
        \lim_{r\to0}\lambda(r)=\lambda(0)=\frac{1}{n+1}<|b|^{-1},
    \]
    we can choose $\theta>0$ such that
    \[
        \zeta:=\sup_{m\in\R^{n+1}, |m|\le \theta} \left(\sup_{x\in\S^n} x^T\Cov_{\Pi(bm)}(v)x\right)
        =\sup_{r\in[0,\theta |b|]}\lambda(r)<|b|^{-1}.
    \]
    Then with the same argument as in Lemma \ref{lem:gradient} we get
    \begin{equation}\label{eq:convexeps}
        m \cdot \nabla J_t (m) \ge (1- |b| \zeta) \abs{m}^2 \quad \text{for all $m\in\R^{n+1}$ with $|m|\le\theta$.}
    \end{equation}
    We already know from \cite[Theorem 4.5 (i)]{sidI} that almost surely $\mu_t\xrightarrow{w}\cU$ (see \eqref{eq:sidI-thm}), so $m_t\to 0$ and hence also $\tilde m_t \to 0$ (see \eqref{eq:mtilde} and Lemma \ref{lem:Qt}). Therefore, there is an almost surely finite random time $\tau=\tau(\theta)$ such that $|\tilde m_t| \le \theta$ for all $t>\tau$, so with \eqref{eq:convexeps} we get
    \begin{equation}\label{eq:convex'}
        -\tilde m_t \cdot \nabla J_t (\tilde m_t) \le - (1- |b| \zeta) \abs{\tilde m_t}^2 \quad \text{for all $t\ge \tau$.}
    \end{equation}
    Also note that
    \begin{equation}\label{eq:convex''}
        -\tilde m_t \cdot \nabla J_t (\tilde m_t) \le C \quad \text{for all $t\in [t_0,\tau)$.}
    \end{equation}
    With
    \[
        0<\alpha< \min\left\{ 1, 2-2|b| \zeta\right\},
    \]
    we can show as in the proof of Lemma \ref{prop:meantozero} that
    \begin{equation}\label{eq:SDI-n'}
        d|\tilde m_t|^2\le \left( -2\tilde m_t \cdot \nabla J_t (\tilde m_t) + C e^{-\alpha t}\right)dt+2 \tilde m_t \cdot dM_t \quad \text{for all $t\ge t_0$.}
    \end{equation}
    If we set $\xi:=2- 2|b| \zeta$, plug \eqref{eq:SDI-n'} into
    \[
        d\left( e^{\xi t}|\tilde m_t|^2 \right)=\xi e^{\xi t}|\tilde m_t|^2dt+ e^{\xi t}d|\tilde m_t|^2,
    \]
    and then apply \eqref{eq:convex'} and \eqref{eq:convex''}, we can estimate
    \begin{align*}
    \begin{split}
        |\tilde m_t|^2
        &\le e^{-\xi (t-t_0)} |\tilde m_{t_0}|^2 + C\int_{t_0}^t e^{-\xi (t-r)-\alpha r} dr 
        + C \int_{t_0}^{\tau} e^{-\xi (t-r)} dr + 2\int_{t_0}^t e^{-\xi (t-r)}\tilde m_r \cdot  dM_r \\
        &\le K e^{-\alpha t} + 2e^{-\xi t}\int_{t_0}^t e^{\xi r}\tilde m_r\cdot dM_r,
    \end{split}
    \end{align*}
    where $K$ is a positive, almost surely finite random number. Now, \eqref{eq:eps2weaksphere} follows from the same line of reasoning as in the last step of the proof of Lemma \ref{prop:meantozero} and in the proof of Proposition \ref{prop:eps2}.
\end{proof}

\section*{Acknowledgements}
The authors thank the anonymous referees for their careful reading of the manuscript and for many helpful suggestions to improve it.

\section*{Funding}
The first author was supported by Deutsche Forschungsgemeinschaft (DFG, German Research Foundation)-Project No. 233630050-TRR 146 Multiscale Simulation Methods for Soft Matter Systems. This research has been conducted within the F\'ed\'eration Parisienne de Mod\'elisation Math\'ematique (FP2M)-CNRS FR 2036. This research has been conducted as part of the project Labex MME-DII (ANR11-LBX-0023-01).




\end{document}